\documentclass[12pt]{amsart}

\usepackage{amssymb}
\usepackage{amsthm}
\usepackage{amsmath}
\usepackage{graphicx}
\usepackage{enumitem}
\usepackage[capitalize]{cleveref}

\theoremstyle{plain}
\newtheorem{thm}{Theorem}
\newtheorem{prop}[thm]{Proposition}
\newtheorem{lem}[thm]{Lemma}
\newtheorem{cor}[thm]{Corollary}

\theoremstyle{definition}
\newtheorem*{defn*}{Definition}

\newcommand{\length}{\mathsf{len}}
\newcommand{\dsup}{d_{\mathrm{sup}}}
\newcommand{\complex}{\mathbb{C}}
\newcommand{\disk}{\mathbb{D}}
\newcommand{\h}{\widehat}
\newcommand{\hull}{\mathrm{Hull}}
\newcommand{\shadow}{\mathrm{Sh}}

\begin{document}

\title{Shortest paths in arbitrary plane domains}
\author{L. C. Hoehn \and L. G. Oversteegen \and E. D. Tymchatyn}
\date{November 30, 2018}

\address[L.\ C.\ Hoehn]{Nipissing University, Department of Computer Science \& Mathematics, 100 College Drive, Box 5002, North Bay, Ontario, Canada, P1B 8L7}
\email{loganh@nipissingu.ca}

\address[L.\ G.\ Oversteegen]{University of Alabama at Birmingham, Department of Mathematics, Birmingham, AL 35294, USA}
\email{overstee@uab.edu}

\address[E.\ D.\ Tymchatyn]{University of Saskatchewan, Department of Mathematics and Statistics, 106 Wiggins road, Saskatoon, Canada, S7N 5E6}
\email{tymchat@math.usask.ca}

\thanks{The first named author was partially supported by NSERC grant RGPIN 435518}
\thanks{The second named author was partially supported by NSF-DMS-1807558}
\thanks{The third named author was partially supported by NSERC grant OGP-0005616}


\begin{abstract}
Let $\Omega$ be a connected open set in the plane and $\gamma: [0,1] \to \overline{\Omega}$ a path such that $\gamma((0,1)) \subset \Omega$.  We show that the path $\gamma$ can be ``pulled tight'' to a unique shortest path which is homotopic to $\gamma$, via a homotopy $h$ with endpoints fixed whose intermediate paths $h_t$, for $t \in [0,1)$, satisfy $h_t((0,1)) \subset \Omega$.  We prove this result even in the case when there is no path of finite Euclidean length homotopic to $\gamma$ under such a homotopy.  For this purpose, we offer three other natural, equivalent notions of a ``shortest'' path.  This work generalizes previous results for simply connected domains with simple closed curve boundaries.
\end{abstract}

\maketitle

\section{Introduction}
\label{sec:introduction}

Bourgin and Renz \cite{BourginRenz1989} proved that given a simply connected plane domain $\Omega$ with simple closed curve boundary, and given any two points $p,q \in \overline{\Omega}$, there exists a unique \emph{shortest} path in $\overline{\Omega}$ which is the uniform limit of paths which (except possibly for their endpoints $p$ and $q$) are contained in $\Omega$.  If there is a rectifiable such curve (i.e.\ one with finite Euclidean length), then by shortest is meant the one with the smallest Euclidean length.  If not, then by shortest is meant \emph{locally shortest}, which means that every subpath not containing the endpoints $p$ and $q$ is of finite Euclidean length and is shortest among all such paths joining the endpoints of the subpath.  A related result for $1$-dimensional locally connected continua is proved in \cite{CannonConnerZastrow2002}.  

In \cref{thm:main1}, we extend the result of Bourgin and Renz to paths in multiply connected domains with arbitrary boundaries.  Along the way, we characterize the concept of a shortest path in up to four different ways, outlined in the subsections below and stated in \cref{thm:main2}.  To state these concepts and our results precisely, we must fix some terminology and notation regarding paths and homotopies.

For notational convenience, we identify the plane $\mathbb{R}^2$ with the complex numbers $\complex$.  Fix a connected open set $\Omega \subset \complex$ and points $p,q \in \overline{\Omega}$.  We consider paths whose range is contained in $\Omega$, except that the endpoints of the path may belong to $\partial \Omega$.  For brevity, we use the abbreviation ``e.p.e.''\ to mean ``except possibly at endpoints''.  Formally, a \emph{path in $\Omega$ (e.p.e.)\ joining $p$ and $q$} is a continuous function $\gamma: [0,1] \to \complex$ such that $\gamma(0) = p$, $\gamma(1) = q$, and $\gamma(s) \in \Omega$ for all $s \in (0,1)$.  When $p,q \in \partial \Omega$, the existence of such a path is equivalent to the statement that $p$ and $q$ are \emph{accessible} from $\Omega$.


We consider homotopies $h: [0,1] \times [0,1] \to \complex$ such that for each $t \in [0,1)$, the path $h_t: [0,1] \to \complex$ defined by $h_t(s) = h(s,t)$ is a path in $\Omega$ (e.p.e.)\ joining $p$ and $q$.  Given a path $\gamma$ in $\Omega$ (e.p.e.)\ joining $p$ and $q$, let $\overline{[\gamma]}$ denote the set of all paths homotopic to $\gamma$ under such homotopies.  Note that paths in $\overline{[\gamma]}$ may meet the boundary $\partial \Omega$ in more than just the endpoints.  Let $[\gamma]$ denote the set of all paths in $\Omega$ (e.p.e.)\ which belong to $\overline{[\gamma]}$.  Note that in spite of the notation, $\overline{[\gamma]}$ is not equal to the topological closure of $[\gamma]$ (e.g.\ in the function space).

The main result of this paper is that if $\gamma$ is a path in $\Omega$ (e.p.e.)\ joining $p$ and $q$, then $\overline{[\gamma]}$ contains a unique \emph{shortest} path.  By ``shortest'', we mean any one of the equivalent notions introduced next and collected in \cref{thm:main2} below.

First, if there is a path in $\overline{[\gamma]}$ of finite Euclidean length, then we may consider the path in $\overline{[\gamma]}$ of smallest Euclidean length.

Second, we adapt Bourgin \& Renz's condition of locally shortest as follows:

\begin{defn*}
\label{defn:locally shortest}
Let $\gamma$ be a path in $\Omega$ (e.p.e.)\ joining $p$ and $q$, and let $\lambda \in \overline{[\gamma]}$.
\begin{itemize}
\item Given $0 < s_1 < s_2 < 1$, we define the class $\overline{[\lambda^\gamma_{[s_1,s_2]}]}$ as follows.  Let $h: [0,1] \times [0,1] \to \overline{\Omega}$ be a homotopy such that $h_0 = \gamma$, $h_1 = \lambda$, and for each $t \in [0,1)$, $h_t$ is a path in $\Omega$ (e.p.e.)\ joining $p$ and $q$.  We use the homotopy $h$ to ``pull off'' the path $\lambda {\upharpoonright}_{[s_1,s_2]}$ (e.p.e.)\ from $\partial \Omega$ as follows.  Define the path $\lambda^\gamma_{[s_1,s_2]}: [s_1,s_2] \to \overline{\Omega}$ by
\[ \lambda^\gamma_{[s_1,s_2]}(s) = h(s, 1 - (s-s_1)(s-s_2)) .\]
See \cref{fig:locally shortest}.  Though this path is defined in terms of the homotopy $h$, the class $\overline{[\lambda^\gamma_{[s_1,s_2]}]}$ does not depend on the choice of $h$ (see \cref{cor:locally shortest well-defined} below).
\item A path $\lambda \in \overline{[\gamma]}$ is called \emph{locally shortest} if for any $0 < s_1 < s_2 < 1$, the path $\lambda {\upharpoonright}_{[s_1,s_2]}$ has finite Euclidean length, and this length is smallest among all paths in $\overline{[\lambda^\gamma_{[s_1,s_2]}]}$.
\end{itemize}
\end{defn*}

\begin{figure}
\begin{center}
\includegraphics{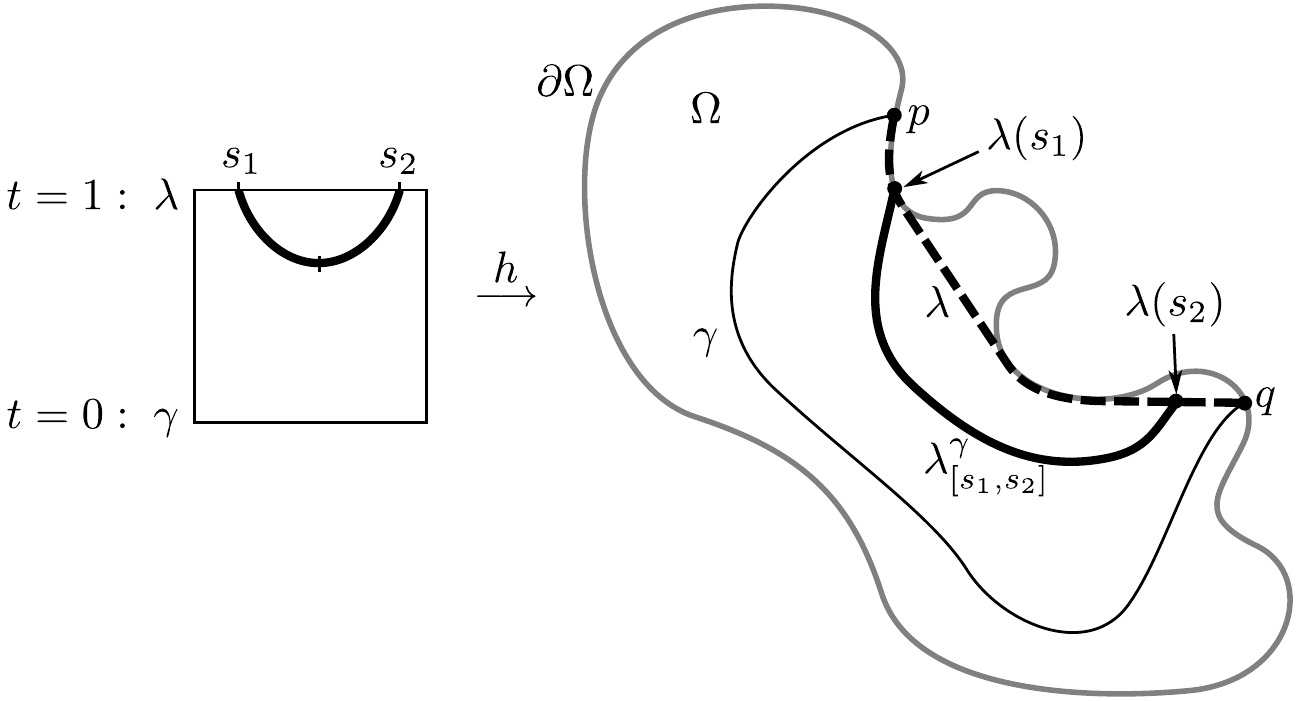}
\end{center}

\caption{An illustration of the paths $\gamma$, $\lambda$, and $\lambda^\gamma_{[s_1,s_2]}$, for the definition of a locally shortest path.}
\label{fig:locally shortest}
\end{figure}

Two further notions of ``shortest path'' are given in the next subsections.

\subsection{Efficient paths}
\label{sec:efficient}

As above, fix a connected open set $\Omega \subset \complex$ and points $p,q \in \overline{\Omega}$.

\begin{defn*}
\label{defn:efficient}
Let $\gamma$ be a path in $\Omega$ (e.p.e.)\ joining $p$ and $q$.  A path $\lambda \in \overline{[\gamma]}$ is called \emph{efficient} (in $\overline{[\gamma]}$) if the following property holds:
\begin{quote}
Given any $s_1,s_2 \in [0,1]$ with $s_1 < s_2$, let $\lambda'$ be the path defined by $\lambda'(s) = \lambda(s)$ for $s \notin [s_1,s_2]$, and $\lambda' {\upharpoonright}_{[s_1,s_2]}$ parameterizes the straight line segment $\overline{\lambda(s_1) \lambda(s_2)}$ or the constant path if $\lambda(s_1) = \lambda(s_2)$.  If $\lambda' \in \overline{[\gamma]}$, then $\lambda = \lambda'$ (up to reparameterization).
\end{quote}
\end{defn*}

The following is the first main result of this paper.  The proof is in \cref{sec:proof main1}.

\begin{thm}
\label{thm:main1}
Let $\Omega \subset \complex$ be a connected open set, let $p,q \in \overline{\Omega}$, and let $\gamma$ be a path in $\Omega$ (except possibly at endpoints) joining $p$ and $q$.  Then $\overline{[\gamma]}$ contains a unique (up to parameterization) efficient path.
\end{thm}

\subsection{Alternative notion of path length}
\label{sec:len}

Instead of the standard Euclidean path length, which can only distinguish between two paths if at least one of them is rectifiable, we can use an alternative notion of path length for which all paths have finite length, and which has other useful properties and many features in common with Euclidean length.  A notion of length with such properties was first introduced in \cite{Morse1936} and further developed in \cite{Silverman1969}.  A simlar notion is given in \cite{CannonConnerZastrow2002}, and the authors of the present paper modified and extended that notion in \cite{HOT2018-1}.

Let $\length$ refer to either the path length function introduced in \cite{HOT2018-1} or in \cite{Morse1936}.  The essential properties of $\length$ are:
\begin{enumerate}
\item $\length(\gamma)$ is finite (in fact $\length(\gamma) < 1$) for any path $\gamma$;
\item $\length(\gamma) = 0$ if and only if $\gamma$ is constant;
\item For any $p,q \in \complex$, the straight line segment $\overline{pq}$ has smallest $\length$ length among all paths from $p$ to $q$.  Moreover, if $\gamma$ is a path from $p$ to $q$ which deviates from the straight line segment $\overline{pq}$, or which is not monotone, then $\length(\gamma) > \length(\overline{pq})$;
\item If $\Phi: \complex \to \complex$ is an isometry, then $\length(\Phi \circ \gamma) = \length(\gamma)$;
\item Given $0 \leq c_1 < c_2 \leq 1$, $\length(\gamma {\upharpoonright}_{[c_1,c_2]}) \leq \length(\gamma)$.  This inequality is strict unless $\gamma$ is constant outside of $[c_1,c_2]$;
\item Given $c \in (0,1)$, $\length(\gamma) \leq \length(\gamma {\upharpoonright}_{[0,c]}) + \length(\gamma {\upharpoonright}_{[c,1]})$.  This inequality is strict unless one of the subpaths is constant;
\item $\length$ is a continuous function from the space of paths (with the uniform metric $\dsup$) to $\mathbb{R}$.
\end{enumerate}

Our second main result is the following.  The proof is in \cref{sec:proof main2}.

\begin{thm}
\label{thm:main2}
Let $\Omega \subset \complex$ be a connected open set, let $p,q \in \overline{\Omega}$, and let $\gamma$ be a path in $\Omega$ (except possibly at endpoints) joining $p$ and $q$.  Then for a path $\lambda \in \overline{[\gamma]}$, the following are equivalent:
\begin{enumerate}
\item $\lambda$ is locally shortest;
\item $\lambda$ is efficient; and
\item $\lambda$ has smallest $\length$ length among all paths in $\overline{[\gamma]}$.
\end{enumerate}
Moreover, if $\overline{[\gamma]}$ contains a rectifiable path, then in addition to the above we have
\begin{enumerate}
\setcounter{enumi}{3}
\item $\lambda$ has smallest Euclidean length among all paths in $\overline{[\gamma]}$.
\end{enumerate}
\end{thm}

\subsection*{Ackowledgements}

The authors would like to thank Professor Alexandre Eremenko for educating them on the history of analytic covering maps and the referee of an earlier version of this paper for several helpful comments which helped to clarify some of the arguments.

\section{Preliminaries on bounded analytic covering maps}
\label{sec:covering maps}

Our arguments in Sections \ref{sec:proof main1} and \ref{sec:proof main2} make heavy use of the theory of complex analytic covering maps.  In this section, we collect the results we will use later.

Denote $\disk = \{z \in \complex: |z| < 1\}$.  It is a standard classical result (see e.g.\ \cite{Ahlfors1973}) that for any connected open set $\Omega \subset \complex$ whose complement contains at least two points, and for any $z_0 \in \Omega$, there is a complex analytic covering map $\varphi: \disk \to \Omega$ such that $\varphi(0) = z_0$.  Moreover, this covering map $\varphi$ is uniquely determined by the argument of $\varphi'(0)$.

Many of the results below hold only for analytic covering maps $\varphi: \disk \to \Omega$ to bounded sets $\Omega$.  For the remainder of this subsection, let $\Omega \subset \complex$ be a bounded connected open set, and let $\varphi: \disk \to \Omega$ be an analytic covering map.

We first state three classic results about the boundary behavior of $\varphi$.  All of the subsequent results in this section will be derived from these and from standard covering map theory.

\begin{thm}[Fatou \cite{Fatou06}, see e.g.\ {\cite[p.22]{Conway1995}}]
\label{thm:Fatou}
The radial limits $\lim_{r \to 1^-} \varphi(r\alpha)$ exist for all points $\alpha \in \partial \disk$ except possibly for a set of linear measure zero in $\partial \disk$.
\end{thm}

It can easily be seen that if $\alpha \in \partial \disk$ is such that $\lim_{r \to 1^-} \varphi(r\alpha)$ exists, then the limit belongs to $\partial \Omega$.

\begin{thm}[Riesz \cite{Riesz16, Riesz23}, see e.g.\ {\cite[p.22]{Conway1995}}]
\label{thm:Riesz}
For each $z \in \partial \Omega$, the set of points $\alpha \in \partial \disk$ for which $\lim_{r\to 1^-} \varphi(r\alpha) = z$ has linear measure zero in $\partial \disk$.
\end{thm}

\begin{thm}[Lindel\"{o}f \cite{Lindelof1915}, see e.g.\ {\cite[p.23]{Conway1995}}]
\label{thm:Lindelof}
Let $\h{\gamma}: [0,1] \to \overline{\disk}$ be a path such that $\h{\gamma}([0,1)) \subset \disk$ and $\h{\gamma}(1) = \alpha \in \partial \disk$.  Suppose that $\lim_{t \to 1^-} \varphi \circ \h{\gamma}(t)$ exists.  Then the radial limit $\lim_{r\to 1^-} \varphi(r\alpha)$ exists and is equal to $\lim_{t \to 1^-} \varphi \circ \h{\gamma}(t)$.
\end{thm}

We define the \emph{extended covering map}
\[ \overline{\varphi}: \disk \cup \{\alpha \in \partial \disk: \lim_{r \to 1^-} \varphi(r\alpha) \textrm{ exists}\} \to \overline{\Omega} \]
by
\[ \overline{\varphi}(\h{z}) = \begin{cases}
\varphi(\h{z}) & \textrm{if } \h{z} \in \disk \\
\lim_{r \to 1^-} \varphi(r\alpha) & \textrm{if } \h{z} = \alpha \in \partial \disk .
\end{cases} \]
By \cref{thm:Fatou}, this function is defined on $\disk$ plus a full measure subset of $\partial \disk$.  Note that this function $\overline{\varphi}$ is not necessarily continuous at points where it is defined in $\partial \disk$.  It is, however, continuous by definition along each radial segment from the center $0$ of $\disk$ to any point $\alpha \in \partial \disk$ where it is defined.  In fact, more is true: if $\overline{\varphi}$ is defined at $\alpha$, then its restriction to any \emph{Stolz angle} at $\alpha$ is continuous (see e.g.\ \cite[p.23]{Conway1995}); however, we will not need this concept in this paper.


As a general convention, we will put hats on symbols, as in $\h{z}$, $\h{A}$, or $\h{\gamma}$, to refer to points, subsets, or paths in $\overline{\disk}$, and use symbols without hats to refer to points, subsets, or paths in $\overline{\Omega} \subset \complex$, with the understanding that if both $z$ and $\h{z}$ appear in an argument, they are related by $z = \overline{\varphi}(\h{z})$ (and likewise for sets and paths).  If $z = \overline{\varphi}(\h{z})$ for points $z \in \overline{\Omega}$ and $\h{z} \in \overline{\disk}$ (respectively, $A = \overline{\varphi}(\h{A})$ for sets $A \subset \overline{\Omega}$ and $\h{A} \subset \overline{\disk}$), we say $\h{z}$ is a \emph{lift} of $z$ (respectively, $\h{A}$ is a \emph{lift} of $A$).  Similarly, if $\gamma = \overline{\varphi} \circ \h{\gamma}$ for paths $\gamma$ in $\overline{\Omega}$ and $\h{\gamma}$ in $\overline{\disk}$, we say $\h{\gamma}$ is a \emph{lift} of $\gamma$.

The next result about lifts of paths is very similar to classical results for covering maps.  Since our extended map $\overline{\varphi}$ is not necessarily continuous at points in $\partial \disk$, it is not a simple consequence of basic covering map theory.  It is proved in \cite{HOT2018-2}, and we also include a proof here to convey the flavor of working with the extended map $\overline{\varphi}$.

\begin{thm}
\label{thm:lift}
Let $p,q \in \overline{\Omega}$, let $z \in \Omega$, and let $\h{z} \in \disk$ such that $\varphi(\h{z}) = z$.
\begin{enumerate}
\item Let $\gamma$ be a path in $\Omega$ (e.p.e.)\ joining $p$ and $q$, and suppose $\gamma(s_0) = z$ for some $s_0 \in [0,1]$.  Then there exists a unique path $\h{\gamma}$ in $\disk$ (e.p.e.)\ such that $\overline{\varphi} \circ \h{\gamma} = \gamma$ and $\h{\gamma}(s_0) = \h{z}$.
\item Let $h: [0,1] \times [0,1] \to \complex$ be a homotopy such that for each $t \in [0,1]$ the path $h_t(s) = h(s,t)$ is a path in $\Omega$ (e.p.e.)\ joining $p$ and $q$, and suppose $h(s_0,0) = z$ for some $s_0 \in [0,1]$.  Then there exists a unique homotopy $\h{h}: [0,1] \times [0,1] \to \overline{\disk}$ such that $\h{h}(s,t) \in \disk$ and $\varphi \circ \h{h}(s,t) = h(s,t)$ for each $(s,t) \in (0,1) \times [0,1]$, $\h{h}(\{0\} \times [0,1])$ is a single point $\h{p}$ in $\overline{\disk}$ with $\overline{\varphi}(\h{p}) = p$, $\h{h}(\{1\} \times [0,1])$ is a single point $\h{q}$ in $\overline{\disk}$ with $\overline{\varphi}(\h{q}) = q$, and $\h{h}(s_0,0) = \h{z}$.
\end{enumerate}
\end{thm}

\begin{proof}
Observe that (1) follows from (2), by using the constant homotopy $h(s,t) = \gamma(s)$ for all $t \in [0,1]$.  Thus it suffices to prove (2).

Since $\varphi$ is a covering map and $h((0,1) \times [0,1]) \subset \Omega$, it follows from standard covering space theory that there is a unique homotopy $\h{h}: (0,1) \times [0,1] \to \disk$ such that $\varphi \circ \h{h}(s,t) = h(s,t)$ for each $(s,t) \in (0,1) \times [0,1]$, and $\h{h}(s_0,0) = \h{z}$.  It remains to prove that there are points $\h{p},\h{q} \in \disk$ such that defining $\h{h}(\{0\} \times [0,1]) = \{\h{p}\}$ and $\h{h}(\{1\} \times [0,1]) = \{\h{q}\}$ makes $\h{h}$ into a continuous function from $[0,1] \times [0,1]$ to $\overline{\disk}$.  This is immediate if $p$ and $q$ belong to $\Omega$, so we assume $p,q \in \partial \Omega$.

Observe that the set $\h{h}((0,\frac{1}{2}] \times [0,1])$ compactifies on a continuum $K \subset \partial \disk$.  We need to prove that $K$ is a single point $\{\h{p}\}$.  Suppose for a contradication that $K$ contains more than one point.  Then there exists by \cref{thm:Fatou} a set $E$ of positive measure in the interior of $K$ so that for each $\alpha \in E$, the radial limit $\lim_{r \to 1^-} \varphi(r\alpha)$ exists.  Since the set $\h{h}((0,\frac{1}{2}] \times [0,1])$ compactifies on $K$ we can choose, for each $\alpha \in E$, a sequence $(s_n,t_n)$ in $(0,\frac{1}{2}] \times [0,1]$ such that $s_n \to 0$ and $\h{h}(s_n,t_n) = r_n \alpha$, with $r_n \to 1$.  It follows that the radial limit $\lim_{r \to 1^-} \varphi(r\alpha) = p$ for each $\alpha \in E$, a contradiction with \cref{thm:Riesz}.  Thus $K$ is a single point $\{\h{p}\}$, and so we can continuously extend $\h{h}$ to $\{0\} \times [0,1]$ by defining $\h{h}(\{0\} \times [0,1]) = \{\h{p}\}$.  By \cref{thm:Lindelof} it follows that $\overline{\varphi}(\h{p}) = p$.

Likewise, by considering $\h{h}([\frac{1}{2},1) \times [0,1])$ we obtain by the same argument a point $\h{q} \in \overline{\disk}$ such that $\h{h}$ extends continuously to $\{1\} \times [0,1]$ by defining $\h{h}(\{1\} \times [0,1]) = \{\h{q}\}$, and $\overline{\varphi}(\h{q}) = q$.
\end{proof}

\cref{thm:lift}(2) implies that if $p,q \in \overline{\Omega}$, $\gamma$ is a path in $\Omega$ (e.p.e.)\ joining $p$ and $q$, $\h{\gamma}$ is a lift of $\gamma$ with endpoints $\h{p},\h{q} \in \overline{\disk}$, and $\lambda \in [\gamma]$, then there exists a lift $\h{\lambda}$ of $\lambda$ with the same endpoints $\h{p},\h{q}$ as $\h{\gamma}$.  We will prove a stronger statement in \cref{thm:class closure lift} below.

The next result follows immediately from \cref{thm:lift}(1).

\begin{cor}
\label{cor:line lift}
Let $L$ be an open arc in $\Omega$ whose closure is an arc with distinct endpoints in $\partial \Omega$.  Let $\h{L}$ be a component of $\varphi^{-1}(L)$.  Then $\h{L}$ is an open arc in $\disk$ whose closure in $\overline{\disk}$ is an arc with distinct endpoints in $\partial \disk$.
\end{cor}

We next prove a result about the existence of small ``crosscuts'' in $\disk$ straddling any point $\alpha \in \partial \disk$ for which $\overline{\varphi}(\alpha) \in \partial \Omega$ is not isolated in $\partial \Omega$.  For one-to-one analytic maps $\varphi: \disk \to \complex$, this result is standard; see e.g.\ \cite{Pommerenke1992}.  We were unable to find a reference for the case of a bounded analytic covering map $\varphi$, so we include a proof for completeness.

\begin{thm}
\label{thm:crosscut}
Let $\alpha \in \partial \disk$ be such that the radial limit $\lim_{r \to 1^-} \varphi(r\alpha)$ exists, and let $p \in \partial \Omega$ be the limit.
\begin{enumerate}
\item If $p$ is not isolated in $\partial \Omega$, then for any sufficiently small simple closed curve $S$ in $\complex$ containing $p$ in its interior, there is a component $\h{S}$ of $\varphi^{-1}(S)$ whose closure is an arc separating $\alpha$ from the center $0$ of $\disk$ in $\overline{\disk}$.
\item If $p$ is isolated in $\partial \Omega$, then for any sufficiently small simple closed curve $S$ in $\Omega$ containing $p$ in its interior, there is a component $\h{S}$ of $\varphi^{-1}(S)$ whose closure is a circle whose intersection with $\partial \disk$ is $\{\alpha\}$.
\end{enumerate}
Moreover, in both cases, the diameter of $\h{S}$ can be made arbitrarily small by choosing $S$ sufficiently small.
\end{thm}

\begin{proof}
Let $S$ be small enough so that $\varphi(0)$ is not in the closed topological disk bounded by $S$.  Let $r_0 < 1$ be close enough to $1$ so that $\varphi(r\alpha)$ is in the interior of $S$ for all $r \in [r_0,1)$.  The closed (in $\disk$) set $\varphi^{-1}(S)$ must separate $r_0 \alpha$ from $0$ in $\disk$, since otherwise there would be a path from $0$ to $r_0 \alpha$ which would project to a path in $\Omega$ from $\varphi(0)$ to $\varphi(r_0 \alpha)$ without intersecting $S$, a contradiction.  Therefore, there must be a component $\h{S}$ of $\varphi^{-1}(S)$ which separates $r_0 \alpha$ from $0$ in $\disk$ (see e.g.\ \cite[p.438 \S 57 III Theorem 1]{Kuratowski1968}).

For (1), suppose that $p$ is not isolated in $\partial \Omega$.  Assume first that $S \cap \partial \Omega \neq \emptyset$.  Let $C$ be the component of $S \cap \Omega$ containing $\varphi(\h{S})$.  This $C$ is a path in $\Omega$ (e.p.e.)\ joining two points (not necessarily distinct) $a,b \in \partial \Omega$ with $a \neq p \neq b$.  By \cref{thm:lift} and the fact that $C$ is an open arc in $\Omega$, we have that each component of $\varphi^{-1}(C)$ is an arc in $\disk$ joining points $\h{a},\h{b} \in \partial \disk$ at which the radial limits exist and are equal to $a$ and $b$, respectively; in particular, we have $\h{a} \neq \alpha \neq \h{b}$.  It follows that the closure of $\h{S}$ is an arc with endpoints distinct from $\alpha$, which separates $\alpha$ from $0$ in $\overline{\disk}$, as desired.

Now choose $\varepsilon_1 > 0$ small enough so that $\varphi(0)$ is not in the closed disk $\overline{B}(p,\varepsilon_1)$ and such that $\partial B(p,\varepsilon_1) \cap \partial \Omega \neq \emptyset$.  Assume that $S \subset B(p,\varepsilon_1)$.  If $S \cap \partial \Omega \neq \emptyset$, then we are done by the previous paragraph; hence, suppose that $S \cap \partial \Omega = \emptyset$.  Choose $\varepsilon_0 > 0$ small enough so that $B(p,\varepsilon_0)$ is contained in the interior of $S$, and such that $\partial B(p,\varepsilon_0) \cap \partial \Omega \neq \emptyset$.  By the previous paragraph, there are components $\h{S}_0$ and $\h{S}_1$ of $\varphi^{-1}(\partial B(p,\varepsilon_0))$ and $\varphi^{-1}(\partial B(p,\varepsilon_1))$, respectively, which are arcs separating $\alpha$ from $0$.  As above, there is a component $\h{S}$ of $\varphi^{-1}(S)$ which separates a tail end of the radial segment at $\alpha$ from $0$ in $\disk$, and which separates $\h{S}_0$ from $\h{S}_1$ in $\disk$ (i.e.\ lies between $\h{S}_0$ and $\h{S}_1$).  This implies the endpoints of the closure of $\h{S}$ are distinct, hence the closure of $\h{S}$ is an arc, as desired.

For (2), let $S$ be small enough so that $p$ is the only point of $\partial \Omega$ in the closed topological disk bounded by $S$.  As above, we obtain a component $\h{S}$ of $\varphi^{-1}(S)$ which separates a tail end of the radial segment at $\alpha$ from $0$ in $\disk$.  Since $\varphi$ is a covering map, this $\h{S}$ must be an open arc in $\disk$ whose two ends both accumulate on continua $K_1$ and $K_2$ in $\partial \disk$.

We first argue that $K_1$ and $K_2$ are single points.  Suppose $\beta_1,\beta_2 \in K_1$ with $\beta_1 \neq \beta_2$.  Then by \cref{thm:Fatou} there exists $\beta \in K_1$ between $\beta_1$ and $\beta_2$ where the radial limit $\lim_{r \to 1^-} \varphi(r\beta)$ exists.  However, this radial segment meets $\h{S}$ arbitrarily close to $\beta$, hence $\lim_{r \to 1^-} \varphi(r\beta) \notin \partial \Omega$, a contradiction.  Thus $K_1$ is a single point $\{\h{a}\}$.  Similarly, $K_2$ is a single point $\{\h{b}\}$.

If $\h{a} \neq \alpha$, by \cref{thm:Fatou} and \cref{thm:Riesz} we can find $\beta \in \partial \disk$ between $\h{a}$ and $\alpha$ such that the radial limit $\lim_{r \to 1^-} \varphi(r\beta)$ exists and is different from $p$.  The path $\overline{\varphi}(r\beta)$, $0 \leq r \leq 1$, is homotopic to one which does not enter the interior of $S$.  By \cref{thm:lift}, we can lift this homotopy, to obtain a path from $0$ to $\beta$ in $\disk$ which does not meet $\h{S}$.  But this is a contradiction since $\h{S}$ separates $\beta$ from $0$.  Therefore $\h{a} = \alpha$, and likewise $\h{b} = \alpha$.  Thus the closure of $\h{S}$ is a circle meeting $\partial \disk$ at $\alpha$ only, as desired.

For the moreover part, we argue as in the proof of \cref{thm:lift} that if these components $\h{S}$ did not converge to $0$ in diameter as the diameter of $S$ is shrunk towards $0$, then they would accumulate on a non-degenerate continuum $K \subset \partial \disk$, and we would obtain a contradiction by \cref{thm:Fatou} and \cref{thm:Riesz}.  The details are left to the reader.
\end{proof}

\begin{lem}
\label{lem:large lifts}
Let $S$ be a straight line or round circle in $\complex$, and let $\varepsilon > 0$.  Then there are only finitely many lifts of components of $S \cap \Omega$ with diameter at least $\varepsilon$.
\end{lem}

\begin{proof}
Suppose the claim is false, so that there exists $\varepsilon > 0$ and infinitely many lifts $\h{S}_n$, $n = 1,2,\ldots$, of components of $S \cap \Omega$ such that the diameter of $\h{S}_n$ is at least $\varepsilon$ for each $n$.  These lifts accumulate on a non-degenerate continuum $K \subset \overline{\disk}$.

If $K \cap \disk \neq \emptyset$, then let $\h{z} \in K \cap \disk$ and let $\h{V}$ be a neighborhood of $\h{z}$ which maps one-to-one under $\varphi$ to a small round disk $V \subset \Omega$.  Then $\h{V}$ meets infinitely many of the lifts $\h{S}_n$.  On the other hand, because $V$ is a round disk contained in $\Omega$ and $S$ is a round circle or straight line, it follows that $V$ can only meet one component of $S \cap \Omega$.  This is a contradiction since $\varphi$ is one-to-one on $\h{V}$.

Suppose then that $K \subset \partial \disk$.  Then there exists by \cref{thm:Fatou} a set $E$ of positive measure in the interior of $K$ so that for each $\alpha \in E$, the radial limit $\lim_{r \to 1^-} \varphi(r\alpha)$ exists.

If there is a single component $S'$ of $S \cap \Omega$ such that $\varphi(\h{S}_n) = S'$ for infinitely many $n$, then it is clear that the radial limit of $\varphi$ at each $\alpha \in E$ must belong to $\overline{S'} \cap \partial \Omega$, which contains at most two points.  But this contradicts \cref{thm:Riesz}.

Therefore we may assume that the components $\varphi(\h{S}_n)$ are all distinct, which means their diameters must converge to $0$.  By passing to a subsequence if necessary, we may assume that the components $\varphi(\h{S}_n)$ converge to a single point $a \in S \cap \partial \Omega$.  Then it is clear that the radial limit of $\varphi$ at each $\alpha \in E$ must equal $a$, again a contradiction by \cref{thm:Riesz}.
\end{proof}

Recall that the function $\overline{\varphi}$ is not necessarily continuous at points $\alpha \in \partial \disk$ where it is defined.  However, the next result shows that the restriction of $\overline{\varphi}$ to the region in between two lifted paths with the same endpoints is continuous.

Given a continuum $X$ in $\complex$, the \emph{topological hull} of $X$, denoted $\hull(X)$, is the smallest simply connected continuum in $\complex$ containing $X$.  Equivalently, $\hull(X)$ is equal to $\complex \smallsetminus U$, where $U$ is the unbounded component of $\complex \smallsetminus X$.

\begin{lem}
\label{lem:hull continuous}
Suppose $\h{\lambda}: [0,1] \to \overline{\disk}$ is a path such that $\overline{\varphi} \circ \h{\lambda}$ is a path in $\overline{\Omega}$ (i.e.\ is continuous).  Let $X$ be the union of $\h{\lambda}([0,1])$ with the two radial segments from the center $0$ of $\disk$ to $\h{p} = \h{\lambda}(0)$ and to $\h{q} = \h{\lambda}(1)$, and let $\Delta = \hull(X)$.  Then $\Delta$ is simply connected and locally connected, and the restriction $\overline{\varphi} {\upharpoonright}_\Delta$ is continuous on $\Delta$.
\end{lem}

\begin{proof}
Note that $\Delta$ is simply connected by definition, and it is straightforward to see that the topological hull of any locally connected continuum is locally connected.  Since $\varphi$ is continuous and $\overline{\varphi} = \varphi$ in $\disk$, it remains to prove that the restriction of $\overline{\varphi}$ to $\Delta$ is continuous at each point of $\Delta \cap \partial \disk$.

Let $\alpha \in \Delta \cap \partial \disk$, and suppose for a contradiction that the restriction of $\overline{\varphi}$ to $\Delta$ is not continuous at $\alpha$.  Then there exists $\varepsilon > 0$ and a sequence of points $\langle \h{w}_n \rangle_{n=1}^\infty$ in $\Delta$ such that $\h{w}_n \to \alpha$, but $|\overline{\varphi}(\alpha) - \varphi(\h{w}_n)| \geq \varepsilon$ for all $n$.

Note that since $\Delta \cap \partial \disk \subset \h{\lambda}([0,1])$ and $\overline{\varphi} \circ \h{\lambda}$ is continuous, we have that the restriction of $\overline{\varphi}$ to $\Delta \cap \partial \disk$ is continuous.  Hence, we may assume that $\h{w}_n \in \Delta \cap \disk$ for each $n$.  Moreover, since the restriction of $\overline{\varphi}$ to the radial segment from $0$ to $\alpha$ is continuous, we may also assume that $\h{w}_n$ does not belong to this segment for all $n$.

For each $n$, there is a lift $\h{C}_n$ of a component of $\Omega \cap \partial B(\overline{\varphi}(\alpha), \varepsilon)$ which separates $\h{w}_n$ from $\alpha$ in $\overline{\disk}$, where $B(\overline{\varphi}(\alpha), \varepsilon)$ is the open disk centered at $\overline{\varphi}(\alpha)$ of radius $\varepsilon$.  By \cref{thm:crosscut}, this lift $\h{C}_n$ is either an arc with endpoints in $\partial \disk$, or, in the case that $\overline{\varphi}(\alpha)$ is the only point of $\partial \Omega$ in the closed disk $\overline{B}(\overline{\varphi}(\alpha), \varepsilon)$, a circle with one point on $\partial \disk$.  By passing to a subsequence if necessary, we may assume that all of the $\h{C}_n$ are distinct, hence they are pairwise disjoint.  According to \cref{lem:large lifts}, the diameters of the lifts $\h{C}_n$ converge to $0$ as $n \to \infty$.  Let $X$ be as in the statement of the lemma, and for each $n$ choose a point $\h{z}_n \in X \cap C_n$.  Then $\h{z}_n \to \alpha$.

Suppose first that there is a subsequence $\h{z}_{n_k}$ of $\h{z}_n$ such that for each $k$, $\h{z}_{n_k}$ belongs to the radial segment from $0$ to $\h{p}$.  It follows that $\alpha = \h{p}$.  Since $\overline{\varphi}$ is continuous on this radial segment, we have $\varphi(\h{z}_{n_k}) \to \overline{\varphi}(\alpha)$ as $k \to \infty$.  But $\varphi(\h{z}_n) \in \partial B(\overline{\varphi}(\alpha), \varepsilon)$ for each $n$, so this is a contradiction.  Likewise, we encounter a contradiction if infinitely many of the points $\h{z}_n$ belong to the radial segment from $0$ to $\h{q}$.

Thus we may assume that all of the points $\h{z}_n$ belong to the set $\h{\lambda}([0,1])$.  Let $s_n \in [0,1]$ such that $\h{z}_n = \h{\lambda}(s_n)$.  By passing to a subsequence if necessary, we may assume that $s_n \to s_\infty$, and $\h{\lambda}(s_\infty) = \alpha$.  Since $\overline{\varphi} \circ \h{\lambda}$ is continuous, it follows that $\overline{\varphi}(\h{z}_n) \to \overline{\varphi}(\alpha)$ as $n \to \infty$.  But again this is a contradiction because $\varphi(\h{z}_n) \in \partial B(\overline{\varphi}(\alpha), \varepsilon)$ for each $n$.
\end{proof}

In the next result, we characterize paths in $\overline{[\gamma]}$ in terms of lifts.

\begin{thm}
\label{thm:class closure lift}
Let $\gamma$ be a path in $\Omega$ (e.p.e.)\ joining $p$ and $q$, and let $\h{\gamma}$ be a lift of $\gamma$ with endpoints $\h{p},\h{q} \in \overline{\disk}$.  If $\lambda \in \overline{[\gamma]}$, then there exists a lift $\h{\lambda}$ of $\lambda$ (to $\overline{\disk}$) with the same endpoints $\h{p},\h{q}$.  Conversely, if $\lambda: [0,1] \to \overline{\Omega}$ is a path joining $p$ and $q$ which has a lift $\h{\lambda}: [0,1] \to \overline{\disk}$ with the same endpoints $\h{p},\h{q}$, then $\lambda \in \overline{[\gamma]}$.
\end{thm}

\begin{proof}
Let $\lambda \in \overline{[\gamma]}$, and let $h$ be a homotopy such that $h_0 = \gamma$, $h_1 = \lambda$, and $h_t$ is a path in $\Omega$ (e.p.e.)\ joining $p$ and $q$ for each $t \in [0,1)$.  For each $s \in [0,1]$, consider the path $t \mapsto h_t(s)$.  Apply \cref{thm:lift} to obtain a lift $\h{h}_t(s)$ such that $\h{h}_0(s) = \h{\gamma}(s)$.  Define $\h{\lambda}: [0,1] \to \overline{\disk}$ by $\h{\lambda}(s) = \h{h}_1(s)$.  By \cref{thm:Lindelof}, we have $\overline{\varphi} \circ \h{\lambda}(s) = \lambda(s)$ for all $s \in [0,1]$, and $\h{\lambda}(0) = \h{p}$ and $\h{\lambda}(1) = \h{q}$.

It remains to prove that $\h{\lambda}$ is continuous.  Let $s \in [0,1]$.  If $\h{\lambda}(s) \in \disk$, then $\h{\lambda}$ is continuous at $s$ by standard covering space theory.  Suppose for a contradiction that $\h{\lambda}(s) \in \partial \disk$ and $\h{\lambda}$ is not continuous at $s$.  We proceed with an argument similar to the one given for \cref{thm:lift}.  The sets $\h{h}([s-\frac{1}{n},s+\frac{1}{n}] \times [1-\frac{1}{n},1)$, $n=1,2,\ldots$, accumulate on a non-degenerate continuum $K \subset \partial \disk$, which contains, by \cref{thm:Fatou}, a set $E$ of positive measure such that the radial limit $\lim_{r \to 1^-} \varphi(r\alpha)$ exists for each $\alpha \in E$.  We can choose, for each $\alpha \in E$, a sequence $(s_n,t_n)$ converging to $(s,1)$ such that $\h{h}(s_n,t_n) = r_n \alpha$, with $r_n \to 1$.  It follows that the radial limit $\lim_{r \to 1^-} \varphi(r\alpha)$ is equal to
\[ \lim_{n \to \infty} \varphi(r_n \alpha) = \lim_{n \to \infty} \varphi \circ \h{h}(s_n,t_n) = \lim_{n \to \infty} h(s_n,t_n) = \lambda(s) ,\]
a contradiction with \cref{thm:Riesz}.  Therefore, $\h{\lambda}$ is a continuous lift of $\lambda$.

Conversely, suppose $\lambda: [0,1] \to \overline{\Omega}$ is a path joining $p$ and $q$ which has a lift $\h{\lambda}: [0,1] \to \overline{\disk}$ with the same endpoints $\h{p},\h{q}$ as $\h{\gamma}$.  Let $\h{c}$ be the path in $\disk$ (e.p.e.)\ such that $\h{c}(0) = \h{p}$, $\h{c}(\frac{1}{2}) = 0$, $\h{c}(1) = \h{q}$, and $\h{c}$ linearly parameterizes the straight segments in between these points.  Let $c = \overline{\varphi} \circ \h{c}$.

Let $X = \h{\lambda}([0,1]) \cup \h{c}([0,1])$, and let $\Delta = \hull(X)$.  Since $\Delta$ is simply connected, it follows that there is a homotopy $\h{h}$ between $\h{c}$ and $\h{\lambda}$ within $\Delta$, such that $\h{h}_0 = \h{c}$, $\h{h}_1 = \h{\lambda}$, and $\h{h}_t$ is a path in $\disk$ (e.p.e.)\ joining $\h{p}$ and $\h{q}$ for all $t \in [0,1)$.  Since $\overline{\varphi}$ is continuous on $\Delta$ by \cref{lem:hull continuous}, the composition $\overline{\varphi} \circ \h{h}$ is a homotopy between $c$ and $\lambda$ which establishes that $\lambda \in \overline{[c]}$.  By the same reasoning, we can show that $\gamma \in [c]$.  Therefore, $\lambda \in \overline{[\gamma]}$.
\end{proof}

We now have the machinery in place to conclude that the class $\overline{[\lambda^\gamma_{[s_1,s_2]}]}$ described in the definition of a locally shortest path is well-defined.

\begin{cor}
\label{cor:locally shortest well-defined}
Let $\gamma$ be a path in $\Omega$ (e.p.e.)\ joining $p$ and $q$, and let $\h{\gamma}$ be a lift of $\gamma$ with endpoints $\h{p},\h{q} \in \overline{\disk}$.  Let $\lambda \in \overline{[\gamma]}$, and let $\h{\lambda}$ be a lift of $\lambda$ (to $\overline{\disk}$) with the same endpoints $\h{p},\h{q}$.  Let $0 < s_1 < s_2 < 1$.  A path $\rho$ belongs to $\overline{[\lambda^\gamma_{[s_1,s_2]}]}$ if and only if there is a lift $\h{\rho}$ of $\rho$ (to $\overline{\disk}$) with endpoints $\h{\lambda}(s_1),\h{\lambda}(s_2)$.

In particular, the definition of the class $\overline{[\lambda^\gamma_{[s_1,s_2]}]}$ is independent of the choice of homotopy $h$ between $\gamma$ and $\lambda$.
\end{cor}

\section{Proof of \cref{thm:main1}}
\label{sec:proof main1}

Let $\Omega \subset \complex$ be a connected open set, $p,q \in \overline{\Omega}$, and let $\gamma$ be a path in $\Omega$ (e.p.e.)\ joining $p$ and $q$.  We may assume that either $p \neq q$, or $\gamma$ is a non-trivial loop, so that $\overline{[\gamma]}$ contains no constant path, since otherwise this is obviously the unique efficient path.

Let $D_\Omega$ be a large round disk in $\complex$ which contains the entire path $\gamma([0,1])$.  We may assume, without loss of generality, that $\Omega$ is contained in $D_\Omega$, hence in particular is a bounded subset of $\complex$.  Indeed, any efficient path in $\overline{[\gamma]}$ with respect to $\Omega$ also belongs to $\overline{[\gamma]}$ with respect to $\Omega \cap D_\Omega$, and vice versa.  The same goes for the other notions of shortest path used in this paper.  Hence, we assume $\Omega \subset D_\Omega$ for the remainder of this paper.

Let $\varphi: \disk \to \Omega$ be an analytic covering map, and let $\overline{\varphi}$ denote the extension of $\varphi$ to those points in $\partial \disk$ where the radial limit is defined, as in \cref{sec:covering maps}.  Choose any lift $\h{\gamma}: [0,1] \to \overline{\disk}$ of $\gamma$ under $\overline{\varphi}$ (see \cref{thm:lift}(1)).  So $\h{\gamma}$ is a path in $\disk$ (e.p.e.)\ and $\overline{\varphi} \circ \h{\gamma} = \gamma$.  Let $\h{p} = \h{\gamma}(0) \in \overline{\disk}$ and $\h{q} = \h{\gamma}(1) \in \overline{\disk}$.  Since $\overline{[\gamma]}$ does not contain a constant path, we have $\h{p} \neq \h{q}$ (cf.\ \cref{thm:class closure lift}).

In \cref{sec:sequence of approximations} and \cref{sec:convergence} below, we will establish the existence of an efficient path in $\overline{[\gamma]}$ via a recursive construction in which we repeatedly replace subpaths by straight line segments, when doing so does not change the homotopy class.  We begin with some preliminary results in \cref{sec:lifts of lines}.

\subsection{Lifts of lines}
\label{sec:lifts of lines}

Throughout this paper, when we use the word \emph{line} we mean straight line in $\complex$.  In this subsection, we consider lines $L$ in $\complex$ which intersect $\Omega$, and lifts of closures of components of $L \cap \Omega$ under $\overline{\varphi}$.  By abuse of terminology, any such lift will be called a \emph{lift of $L$}.  For a given line $L$, $L \cap \Omega$ has at most countably many components, and each of these components has at most countably many lifts, each of which is, by \cref{cor:line lift}, an arc in $\overline{\disk}$ whose (distinct) endpoints are in $\partial \disk$, and which is otherwise contained in $\disk$.  Observe that if $\h{L}_1$ and $\h{L}_2$ are distinct lifts of lines, then $\h{L}_1 \cap \h{L}_2$ contains at most one point; moreover, if $\h{L}_1 \cap \h{L}_2 = \{\h{z}\}$ for some $\h{z} \in \disk$, then $\h{L}_1$ and $\h{L}_2$ cross transversally at $\h{z}$.

Our construction later in this section of an efficient path is based on the following reformulation of the definition of an efficient path.

\begin{prop}
\label{prop:efficient characterization}
Let $\lambda \in \overline{[\gamma]}$, and let $\h{\lambda}$ be a lift of $\lambda$ (under $\overline{\varphi}$) joining $\h{p}$ to $\h{q}$.  Then $\lambda$ is an efficient path in $\overline{[\gamma]}$ if and only if for any lift $\h{L}$ of a line intersecting $\Omega$, the set $\h{\lambda}^{-1}(\h{L})$ is connected (possibly empty).
\end{prop}

We next make a detailed study of lifts of lines, focusing on whether they separate $\h{p}$ from $\h{q}$ in $\overline{\disk}$ or not.

\begin{defn*}
Let $L$ be a line which intersects $\Omega$ and which does not contain $p$ or $q$, and let $\h{L}$ be a lift of $L$.
\begin{itemize}
\item We call $\h{L}$ a \emph{separating lift} if it separates $\h{p}$ from $\h{q}$ in $\overline{\disk}$.  The component of $\disk \smallsetminus \h{L}$ whose closure contains $\h{p}$ is called the \emph{$\h{p}$-side} of $\h{L}$, and the component of $\disk \smallsetminus \h{L}$ whose closure contains $\h{q}$ is called the \emph{$\h{q}$-side} of $\h{L}$.
\item We call $\h{L}$ a \emph{non-separating lift} if it does not separate $\h{p}$ from $\h{q}$ in $\overline{\disk}$.  The component of $\disk \smallsetminus \h{L}$ whose closure does not contain $\h{p},\h{q}$ is called the \emph{shadow} of $\h{L}$, denoted $\shadow(\h{L})$.
\end{itemize}
\end{defn*}

Given two lines $L_1,L_2$ which intersect $\Omega$, the \emph{distance} between $L_1$ and $L_2$ is the Hausdorff distance between $L_1 \cap \overline{D_\Omega}$ and $L_2 \cap \overline{D_\Omega}$ (recall that $D_\Omega$ is a fixed large disk in $\complex$ containing $\Omega$); that is, the infimum of all $\delta > 0$ such that each point of $L_1 \cap \overline{D_\Omega}$ is within $\delta$ of a point in $L_2 \cap \overline{D_\Omega}$, and vice versa.

\begin{defn*}
\begin{itemize}
\item Let $\h{V} \subset \disk$ be an open set which maps one-to-one under $\varphi$ to an open set $V \subset \Omega$, and let $L$ be a line which intersects $V$.  Let $\varepsilon > 0$, and assume $\varepsilon$ is small enough so that every line $L'$ which is within distance $\varepsilon$ of $L$ must also intersect $V$.  The family of all lifts $\h{L}'$, of such lines $L'$, which intersect $\h{V}$ will be called a \emph{basic open set of lifts of lines}, and denoted $\h{\mathcal{N}}(\h{L},\h{V},\varepsilon)$.
\item Let $\h{z} \in \disk$.  We say $\h{z}$ is a \emph{stable point} if there exists a basic open set of lifts of lines $\h{\mathcal{N}}(\h{L},\h{V},\varepsilon)$, each element of which is non-separating and contains $\h{z}$ in its shadow.
\end{itemize}
\end{defn*}

\begin{lem}
\label{lem:two lifts}
A point $\h{z} \in \disk$ is a stable point if and only if there are two intersecting non-separating lifts $\h{L}_1,\h{L}_2$ of distinct lines $L_1,L_2$ such that $\h{z}$ is in the shadow of both $\h{L}_1$ and $\h{L}_2$.

Furthermore, whenever $\h{L}_1,\h{L}_2$ are intersecting non-separating lifts of distinct lines, the point of intersection of $\h{L}_1$ and $\h{L}_2$ is also a stable point.
\end{lem}

\begin{proof}
That any stable point has this property is immediate, since any basic open set of lifts of lines clearly contains pairs of distinct intersecting lifts.

Conversely, suppose $\h{L}_1,\h{L}_2$ are intersecting non-separating lifts of distinct lines $L_1 \supset \varphi(\h{L}_1), L_2 \supset \varphi(\h{L}_2)$, and let $\h{z}$ be any point in the shadow of both $\h{L}_1$ and $\h{L}_2$.  Let $\h{w}$ be the point of intersection of $\h{L}_1$ and $\h{L}_2$, and let $\h{V}$ be a neighborhood of $\h{w}$ which maps one-to-one under $\varphi$ to a round disk $V$ centered at $w = \varphi(\h{w})$.

\begin{figure}
\begin{center}
\includegraphics{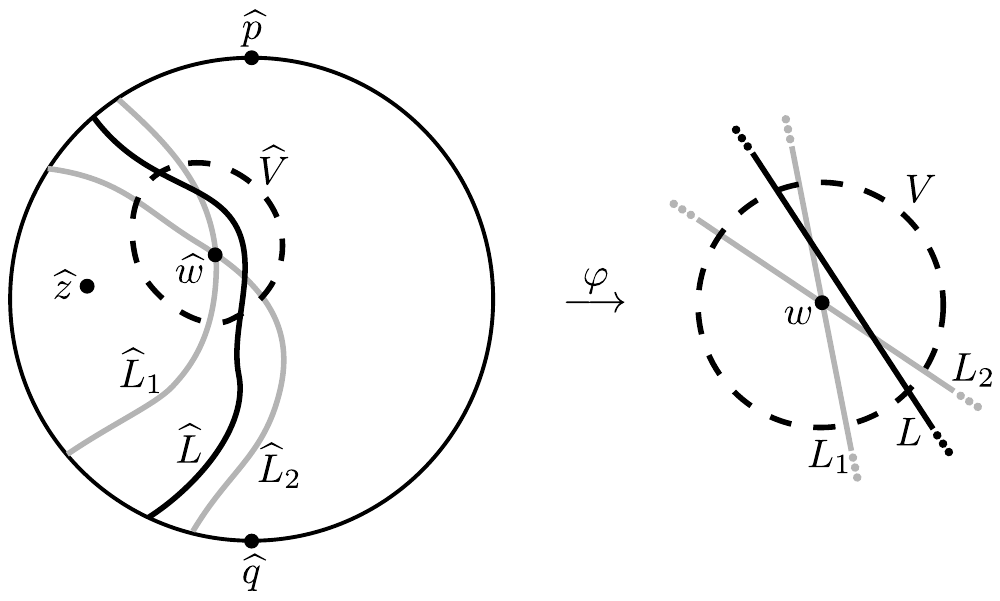}
\end{center}

\caption{The configuration of lines described in the proof of \cref{lem:two lifts}.}
\label{fig:two lifts}
\end{figure}

Consider a line $L$ not containing $w$ which intersects $L_1 \cap V$ and $L_2 \cap V$, and such that the lift $\h{L}$ of $L$ which intersects $\h{V}$ does not contain any point in the intersection of the shadows of $\h{L}_1$ and of $\h{L}_2$.  Let $\varepsilon > 0$ be small enough so that any line $L'$ within distance $\varepsilon$ of $L$ has these same properties.  Let $\h{L}'$ be an arbitrary element of the basic open set of lifts of lines $\h{\mathcal{N}}(\h{L},\h{V},\varepsilon)$; that is, $\h{L}'$ is the lift of a line $L'$ within distance $\varepsilon$ of $L$ such that $\h{L} \cap \h{V} \neq \emptyset$.  Since $\h{L}'$ intersects $\h{L}_1$ and $\h{L}_2$ inside $\h{V}$, it cannot cross them again, and so one endpoint of $\h{L}'$ is in the closure of $\shadow(\h{L}_1) \smallsetminus \shadow(\h{L}_2)$ and the other is in the closure of $\shadow(\h{L}_2) \smallsetminus \shadow(\h{L}_1)$ (see \cref{fig:two lifts}).  It follows that $\h{L}'$ is a non-separating lift and $\shadow(\h{L}')$ contains $\shadow(\h{L}_1) \cap \shadow(\h{L}_2)$ (in particular, it contains the point $\h{z}$), as well as the point $\h{w}$.  Thus $\h{z}$ and $\h{w}$ are both stable points.
\end{proof}

\begin{lem}
\label{lem:stable generic}
The set of stable points is a dense open subset of $\disk$.

Moreover, except for a countable set of lines, every line $L$ has the property that for each lift $\h{L}$ of $L$, the set of stable points in $\h{L}$ is a dense open subset of $\h{L}$.
\end{lem}

\begin{proof}
It follows immediately from \cref{lem:two lifts} that the set of stable points is open in $\disk$.

Now let $\h{z} \in \disk$, and let $\h{V}$ be a neighborhood of $\h{z}$.  Suppose $\h{z}$ is not a stable point.  By shrinking $\h{V}$, we may assume that $\h{V}$ maps one-to-one under $\varphi$ to a disk $V$ centered at $z = \varphi(\h{z})$.

For a given $\theta \in \mathbb{R}$, let $L(z,\theta)$ denote the straight line through $z$ making angle $\theta$ with the positive real axis.  Let $\h{L}(\h{z},\theta)$ be the lift containing $\h{z}$ of the component of $L(z,\theta) \cap \Omega$ containing $z$.  Denote the endpoints of $\h{L}(\h{z},\theta)$ by $e^+(\h{z},\theta)$ and $e^-(\h{z},\theta)$, where $e^+(\h{z},\theta)$ corresponds to following the line $L(z,\theta)$ to $\partial \Omega$ in the direction $\theta$ from $z$, and $e^-(\h{z},\theta)$ corresponds to following the line $L(z,\theta)$ to $\partial \Omega$ in the direction $\theta + \pi$ from $z$.  As $\theta$ increases, the line $L(z,\theta)$ revolves about the point $z$.  Correspondingly, the lift $\h{L}(\h{z},\theta)$ ``revolves'' about $\h{z}$.  The endpoints of $\h{L}(\h{z},\theta)$ move monotonically in the circle $\partial \disk$, but not necessarily continuously, as $\theta$ increases.

Since $\h{z}$ is not stable, by \cref{lem:two lifts} there can be at most one non-separating lift of a line which contains $\h{z}$.  It follows that if we consider the line $L(z,\theta)$ and increase $\theta$ to revolve the line about $z$, at some moment the endpoint $e^+(\h{z},\theta)$ must cross or ``jump over'' $\h{p}$, and at that same moment $e^-(\h{z},\theta)$ must cross or ``jump over'' $\h{q}$.  That is, there exists $\theta_0$ such that for all $\alpha,\beta$ sufficiently close to $\theta_0$ with $\alpha < \theta_0 < \beta$, we have that $\h{p}$ is on the ``left'' side of the arc $\h{L}(\h{z},\alpha)$ (thinking of this arc as oriented from $e^-(\h{z},\alpha)$ to $e^+(\h{z},\alpha)$) and $\h{q}$ is on the ``right'' side, and $\h{p}$ is on the ``right'' side of $\h{L}(\h{z},\beta)$ and $\h{q}$ is on the ``left'' side (see \cref{fig:stable dense}).

Let $\h{w}$ be any point in $\h{V} \smallsetminus \h{L}(\h{z},\theta_0)$ and let $w = \varphi(\h{w})$.  Let $\alpha,\beta$ be as above and sufficiently close to $\theta_0$ so that the lines $L(w,\alpha)$ and $L(w,\beta)$ do not intersect either of the lines $L(z,\alpha)$ or $L(z,\beta)$ inside $\overline{D_\Omega}$ (recall that $D_\Omega$ is a fixed large disk in $\complex$ containing $\Omega$).  This means that the lifts $\h{L}(\h{w},\alpha)$ and $\h{L}(\h{w},\beta)$ do not cross either of the lifts $\h{L}(\h{z},\alpha)$ or $\h{L}(\h{z},\beta)$.  Because of the locations of $\h{p}$ and $\h{q}$ with respect to the lines $\h{L}(\h{z},\alpha)$ and $\h{L}(\h{z},\beta)$, it follows that $\h{L}(\h{w},\alpha)$ and $\h{L}(\h{w},\beta)$ are both non-separating lifts (see \cref{fig:stable dense}).  Hence, by \cref{lem:two lifts}, $\h{w}$ is a stable point.  Thus, the set of stable points is dense in $\disk$.

\begin{figure}
\begin{center}
\includegraphics{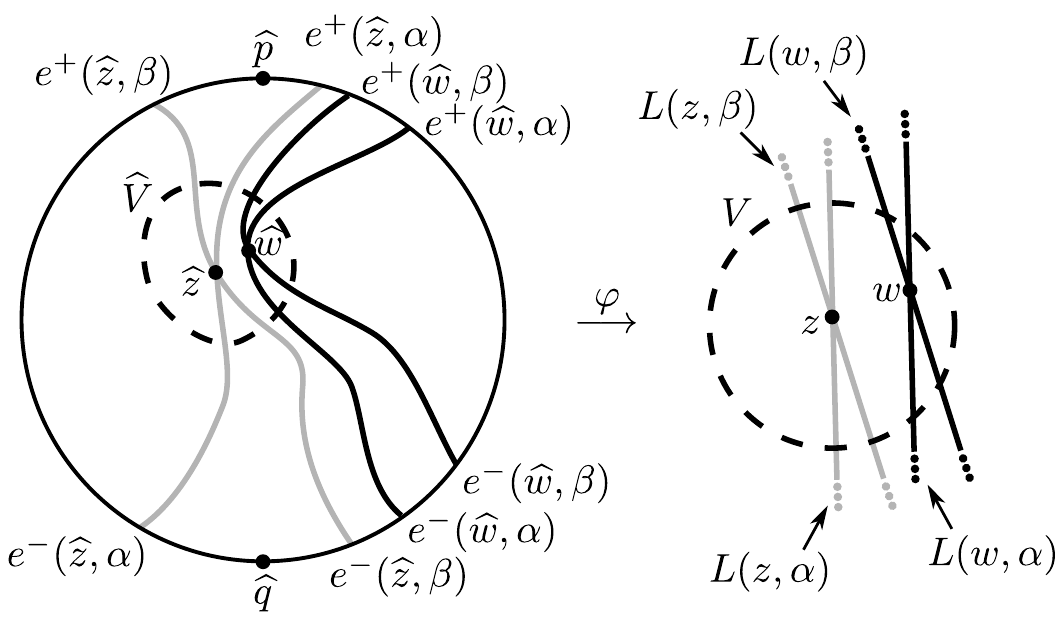}
\end{center}

\caption{The configuration of lines described in the proof of \cref{lem:stable generic}.}
\label{fig:stable dense}
\end{figure}

For the second statement, we may apply the above argument at each $\h{z} \in \disk$ to obtain a countable cover $\{\h{V}_i: i = 1,2,\ldots\}$ of $\disk$ by open sets with the property that for each $i$ there is at most one lift of a line, $\h{L}_i$, such that every point of $\h{V}_i \smallsetminus \h{L}_i$ is stable.  Consider the countable family of lines $\{\varphi(\h{L}_i): i = 1,2,\ldots\}$.  Let $L$ be any line not in this family, let $\h{L}$ be any lift of $L$, and let $\h{z} \in \h{L}$.  Choose $i$ so that $\h{z} \in \h{V}_i$.  If $\h{z}$ is not stable, it must be the (unique) point of intersection of $\h{L}$ and $\h{L}_i$, hence every other point in $\h{L} \cap \h{V}_i$ is stable.  Therefore the set of stable points in $\h{L}$ is a dense open subset of $\h{L}$.
\end{proof}

We remark that since in the proof of \cref{lem:stable generic} the point $\h{w}$ could be chosen on either side of $\h{L}(\h{z},\theta_0)$, it follows from \cref{prop:efficient characterization} and the proof of \cref{lem:stable generic} that if $\h{\gamma}_0$ is a lift of an efficient path in $\overline{[\gamma]}$ with endpoints $\h{p}$ and $\h{q}$, then $\h{\gamma}_0([0,1]) \cap \disk$ coincides with the set of non-stable points.  This observation will not be used in the proof of \cref{thm:main1}.

\subsection{Sequence of approximations of the efficient path}
\label{sec:sequence of approximations}

Let $\mathcal{L}$ be a countable dense (in the sense of Hausdorff distanct) family of distinct straight lines which intersect $\Omega$ and do not contain $p$ or $q$, and such that for each lift $\h{L}$ of a line $L \in \mathcal{L}$ the set of stable points in $\h{L}$ is a dense open subset of $\h{L}$ (this is possible by \cref{lem:stable generic}).  Let $\h{\mathcal{L}}$ denote the (countable) set of all lifts of lines in $\mathcal{L}$.  We enumerate the elements of this set: $\h{\mathcal{L}} = {\langle \h{L}_i \rangle}_{i=1}^\infty$.

We construct a sequence of paths $\gamma_i$, $i \geq 1$, by recursion.  To begin, let $\gamma_1 = \gamma$ and $\h{\gamma}_1 = \h{\gamma}$.  Having defined $\gamma_i$ and its lift $\h{\gamma}_i$, we define $\gamma_{i+1}$ and $\h{\gamma}_{i+1}$ as follows:

\begin{list}{\textbullet}{\leftmargin=2em \itemindent=0em}
\item If $\h{\gamma}_i^{-1}(\h{L}_i)$ has cardinality $\leq 1$, then put $\gamma_{i+1} = \gamma_i$ and $\h{\gamma}_{i+1} = \h{\gamma}_i$.  Otherwise, let $s_1$ and $s_2$ be the smallest and largest (respectively) $s \in [0,1]$ such that $\h{\gamma}_i(s) \in \h{L}_i$.  Let $\h{\gamma}_{i+1}$ be the path in $\disk$ (e.p.e.)\ defined by $\h{\gamma}_{i+1}(s) = \h{\gamma}_i(s)$ for $s \notin [s_1,s_2]$, and $\h{\gamma}_{i+1} {\upharpoonright}_{[s_1,s_2]}$ parameterizes the subarc of $\h{L}_i$ with endpoints $\h{\gamma}_i(s_1)$ and $\h{\gamma}_i(s_2)$ (or $\h{\gamma}_{i+1} {\upharpoonright}_{[s_1,s_2]}$ is constantly equal to $\h{w}$ if $\h{\gamma}_i(s_1) = \h{\gamma}_i(s_2) = \h{w}$).  Let $\gamma_{i+1} = \overline{\varphi} \circ \h{\gamma}_{i+1}$.
\end{list}

\begin{lem}
\label{lem:connected intersection}
Let $\widehat{L}$ be a lift of a line $L$.  If $i$ is such that $\h{\gamma}_i^{-1}(\h{L}) \subset [0,1]$ is connected (respectively, empty), then $\h{\gamma}_j^{-1}(\h{L})$ is connected (respectively, empty) for all $j \geq i$.

In particular, for all $j > i \geq 1$, $\h{\gamma}_j^{-1}(\h{L}_i)$ is connected (or empty).
\end{lem}

\begin{proof}
Let $\widehat{L}$ be a lift of a line $L$, and fix $i$.  We will argue that if $\h{\gamma}_i^{-1}(\h{L}) \subset [0,1]$ is connected (respectively, empty), then $\h{\gamma}_{i+1}^{-1}(\h{L})$ is connected (respectively, empty).  The claim then follows by induction.  To this end, given the definition of $\h{\gamma}_{i+1}$, clearly we may assume $\h{L} \neq \h{L}_i$, so that either $\h{L} \cap \h{L}_i \cap \disk = \emptyset$ or $\h{L}$ and $\h{L}_i$ cross transversally in $\disk$.

Assume first that $\h{\gamma}_i([0,1]) \cap \h{L} = \emptyset$.  This means in particular that $\h{p},\h{q} \notin \h{L}$, $\h{L}$ is a non-separating lift, and $\h{\gamma}_i([0,1])$ is disjoint from the shadow of $\h{L}$.  If $\h{L}_i$ does not meet $\h{L}$ in $\disk$, then clearly $\h{\gamma}_{i+1}([0,1]) \cap \h{L} = \emptyset$ as well.  If $\h{L}_i$ meets $\h{L}$ in $\disk$, then these two lifts cross transversally.  If $\h{\gamma}_{i+1}([0,1])$ meets $\h{L}$, we must have that $\h{\gamma}_{i+1}([s_1,s_2]) \cap \h{L} \neq \emptyset$ where $s_1 = \min\{s \in [0,1]: \h{\gamma}_i(s) \in \h{L}_i\}$ and $s_2 = \max\{s \in [0,1]: \h{\gamma}_i(s) \in \h{L}_i\}$, since $\h{\gamma}_{i+1} = \h{\gamma}_i$ outside of $[s_1,s_2]$.  Since $\h{\gamma}_{i+1}([s_1,s_2])$ parameterizes a subarc of $\h{L}_i$, which crosses $\h{L}$ transversally, it follows that either $\h{\gamma}_{i+1}(s_1)$ or $\h{\gamma}_{i+1}(s_2)$ is in the shadow of $\h{L}$.  But $\h{\gamma}_{i+1}(s_1) = \h{\gamma}_i(s_1)$ and $\h{\gamma}_{i+1}(s_2) = \h{\gamma}_i(s_2)$, so this contradicts the assumption that $\h{\gamma}_i([0,1]) \cap \h{L} = \emptyset$.  Thus $\h{\gamma}_{i+1}([0,1]) \cap \h{L} = \emptyset$, as desired.

Now assume that $\h{\gamma}_i([0,1]) \cap \h{L} \neq \emptyset$, and $\h{\gamma}_i^{-1}(\h{L})$ is an interval $[a,b]$ (here we allow the possibility that $a = b$, which means $[a,b] = \{a\}$).  Suppose $\h{\gamma}_{i+1}$ is obtained from $\h{\gamma}_i$ by replacing the subpath of $\h{\gamma}_i$ from $s_1$ to $s_2$ ($0 < s_1 < s_2 < 1$) with a segment of $\h{L}_i$.  We leave it to the reader to confirm the following claims:
\begin{itemize}
\item If $s_2 \leq a$, then $\h{\gamma}_{i+1}^{-1}(\h{L}) = [a,b]$;
\item If $s_1 \leq a \leq s_2 \leq b$, then $\h{\gamma}_{i+1}^{-1}(\h{L}) = [s_2,b]$;
\item If $s_1 \leq a \leq b \leq s_2$, then $\h{\gamma}_{i+1}^{-1}(\h{L})$ is either empty (if the section $\h{\gamma}_{i+1}([s_1,s_2])$ of $\h{L}_i$ does not meet $\h{L}$), or consists of the single point where the section $\h{\gamma}_{i+1}([s_1,s_2])$ of $\h{L}_i$ crosses $\h{L}$;
\item It is not possible that $a \leq s_1 < s_2 \leq b$, since $\h{L} \neq \h{L}_i$ (unless $\h{\gamma}_i$ is constant on $[s_1,s_2]$, in which case $\h{\gamma}_{i+1} = \h{\gamma}_i$);
\item If $a \leq s_1 \leq b \leq s_2$, then $\h{\gamma}_{i+1}^{-1}(\h{L}) = [a,s_1]$; and
\item If $b \leq s_1$, then $\h{\gamma}_{i+1}^{-1}(\h{L}) = [a,b]$.
\end{itemize}
In all cases, $\h{\gamma}_{i+1}^{-1}(\h{L})$ is connected (possibly empty), as desired.
\end{proof}

\subsection{Convergence to the efficient path}
\label{sec:convergence}

For each $n = 1,2,\ldots$, we choose sets $\mathcal{G}_n$ of lines in $\mathcal{L}$ with the following properties:
\begin{itemize}
\item $\mathcal{G}_n$ is a finite subset of $\mathcal{L}$ for each $n$;
\item $\mathcal{G}_n \subset \mathcal{G}_{n+1}$ for each $n$;
\item each component of $\Omega \smallsetminus \bigcup \mathcal{G}_n$ has diameter less than $\frac{1}{n}$; and
\item if $G_1,G_2 \in \mathcal{G}_n$ are distinct lines and $\h{G}_1$ and $\h{G}_2$ are separating lifts of $G_1$ and $G_2$, respectively, then $\h{G}_1$ and $\h{G}_2$ are either disjoint or meet in a stable point.
\end{itemize}

Let $\h{\mathcal{G}}_n$ be the set of all lifts of lines in $\mathcal{G}_n$.  By \cref{lem:large lifts} only finitely many of the elements of $\h{\mathcal{G}}_n$ intersect the original lifted path $\h{\gamma}_1$, and among them are those that are separating lifts.  We denote the set of all separating lifts in $\h{\mathcal{G}}_n$ by $\h{\mathcal{G}}_n^s$.

We define an order $<$ on $\h{\mathcal{G}}_n^s$ as follows.  Let $\h{G}_1,\h{G}_2 \in \h{\mathcal{G}}_n^s$ be distinct separating lifts.

First suppose $\h{G}_1 \cap \h{G}_2 = \emptyset$.  Then $\h{G}_1 < \h{G}_2$ if $\h{G}_2$ is on the $\h{q}$-side of $\h{G}_1$; otherwise $\h{G}_2 < \h{G}_1$.

Now suppose $\h{G}_1 \cap \h{G}_2 = \{\h{z}\}$ for some stable point $\h{z} \in \disk$.  Since $\h{z}$ is stable, it follows from density of the family of lines $\mathcal{L}$ that there exists $\h{W} \in \h{\mathcal{L}}$ such that $\h{W}$ is non-separating and contains $\h{z}$ in its shadow.  Then $\h{G}_1 < \h{G}_2$ if $\h{G}_2 \smallsetminus \overline{\shadow(\h{W})}$ is on the $\h{q}$-side of $\h{G}_1$; otherwise $\h{G}_2 < \h{G}_1$.

It is straightforward to see that this relation $<$ is a well-defined linear order on $\h{\mathcal{G}}_n^s$.

For each $n = 1,2,\ldots$, choose a finite set $\h{\mathcal{W}}_n$ of lifts of lines from $\h{\mathcal{L}}$ such that for each pair of distinct intersecting separating lifts $\h{G}_1,\h{G}_2 \in \h{\mathcal{G}}_n^s$ there is a lift $\h{W} \in \h{\mathcal{W}}_n$ witnessing that $\h{G}_1 < \h{G}_2$ or that $\h{G}_2 < \h{G}_1$ (i.e.\ such that $\shadow(\h{W})$ contains $\h{G}_1 \cap \h{G}_2$).  Let $i(n)$ be large enough so that $\h{\mathcal{W}}_n \subset \{\h{L}_i: i = 1,\ldots,i(n)-1\}$, and also all of the (finitely many) elements of $\h{\mathcal{G}}_n$ which intersect the original lifted path $\h{\gamma}_1$ are contained in $\{\h{L}_i: i = 1,\ldots,i(n)-1\}$.

Take an arbitrary $n$, and let $\h{G}_1,\ldots,\h{G}_m$ enumerate the elements of $\h{\mathcal{G}}_n^s$, listed in $<$ order.  For any $i \geq i(n)$, by \cref{lem:connected intersection} the lifted path $\h{\gamma}_i$ does not cross any of the non-separating lifts in $\h{\mathcal{G}}_n \cup \h{\mathcal{W}}_n$, and has connected intersection with all the lifts that it meets, in particular with the elements of $\h{\mathcal{G}}_n^s$.  Thus there are $m+1$ components $\h{R}_1,\ldots,\h{R}_{m+1}$ of $\disk \smallsetminus \bigcup (\h{\mathcal{G}}_n \cup \h{\mathcal{W}}_n)$ such that for any $i \geq i(n)$, the lifted path $\h{\gamma}_i$ starts at $\h{p}$ then runs in $\h{R}_1$, until it crosses $\h{G}_1$ and enters $\h{R}_2$, then crosses $\h{G}_2$ and enters $\h{R}_3$, and so on, until it crosses $\h{G}_m$ to enter $\h{R}_{m+1}$ and ends at $\h{q}$ (see \cref{fig:regions}).

\begin{figure}
\begin{center}
\includegraphics{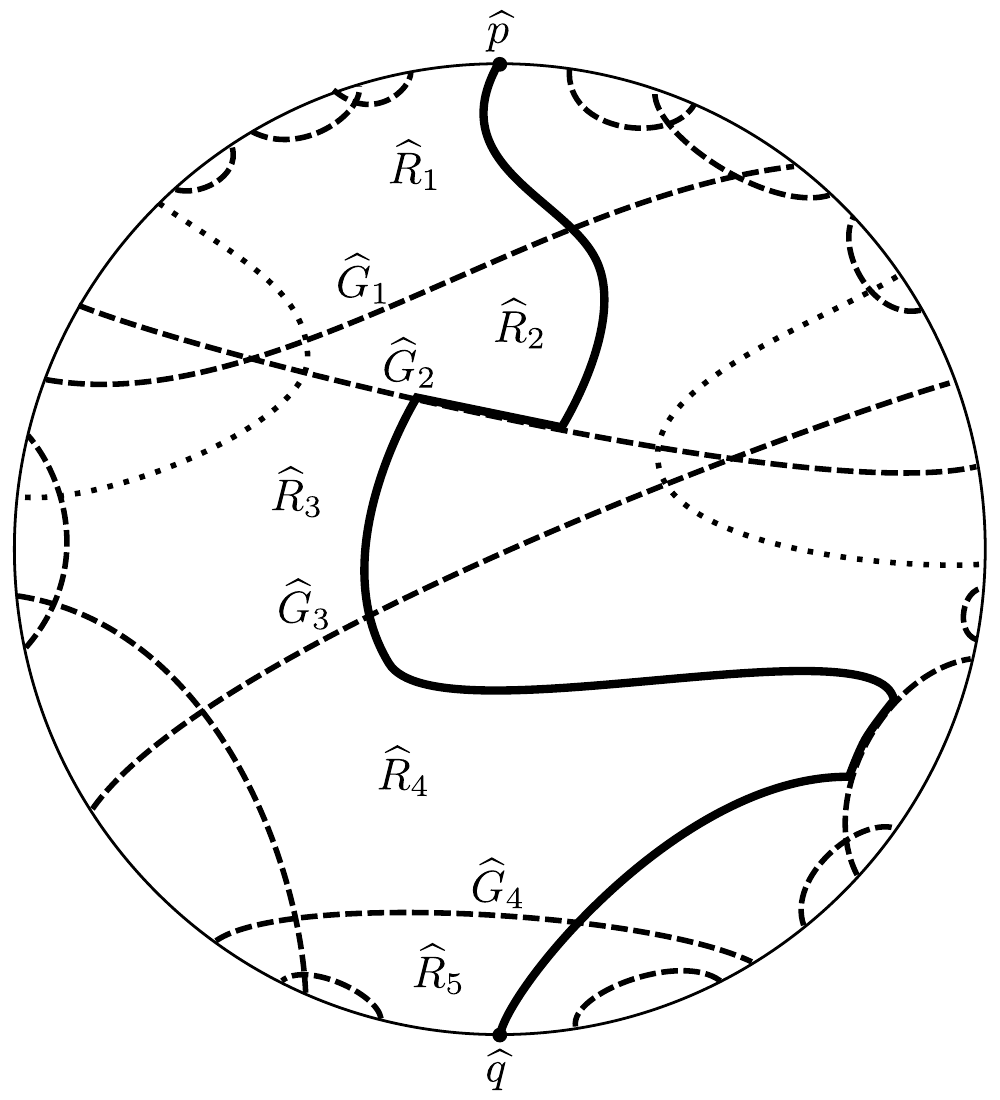}
\end{center}

\caption{An illustration depicting a possible configuration of the lifts of lines $\h{G}_k$ and regions $\h{R}_k$ in between, as described in \cref{sec:convergence}.  The thick curve represents a path $\h{\gamma}_i$ (for $i \geq i(n)$), the dashed arcs represent elements of $\h{\mathcal{G}}_n$ (including the separating lifts $\h{G}_k$), and the dotted arcs represent elements of $\h{\mathcal{W}}_n$.}
\label{fig:regions}
\end{figure}

For each of these regions $\h{R}_j$, $\varphi(\h{R}_j)$ is contained in a component of $\Omega \smallsetminus \bigcup \mathcal{G}_n$, hence has diameter less than $\frac{1}{n}$.  Therefore, provided we parameterize the path $\gamma_i$ so that if $\h{\gamma}_{i(n)}(t) \in \h{R}_j$ then $\h{\gamma}_i(t) \in \h{R}_j$ for all $i \geq i(n)$, we have that $\dsup(\gamma_{i_1},\gamma_{i_2}) < \frac{1}{n}$ for all $i_1,i_2 \geq i(n)$; in fact there is a homotopy between $\gamma_{i_1}$ and $\gamma_{i_2}$ which moves no point more than $\frac{1}{n}$, obtained by moving $\h{\gamma}_{i_1}(t)$ to $\h{\gamma}_{i_2}(t)$ within the closure of a region $\h{R}_j$ which contains $\h{\gamma}_{i(n)}(t)$.

We remark that according to \cref{prop:efficient characterization}, if $\lambda \in \overline{[\gamma]}$ is any efficient path, it must follow this same pattern traversing through the (closures of the) regions $\h{R}_1,\ldots,\h{R}_{m+1}$ in the same monotone order.  Therefore, any two such efficient paths, if parameterized appropriately, are within $\frac{1}{n}$ of each other.  Since $n$ is arbitrary, this establishes that there can be at most one efficient path in $\overline{[\gamma]}$ (up to reparameterization).

It remains to confirm that our construction above yields an efficient path in $\overline{[\gamma]}$.  Since $n$ was arbitrary, we now have that $\langle \gamma_i \rangle_{i=1}^\infty$ is a Cauchy sequence of paths, hence it converges to a path $\gamma_\infty$ from $p$ to $q$.  Moreover, by putting together the (smaller and smaller) homotopies between the paths in this sequence, we have that $\gamma_\infty \in \overline{[\gamma]}$.

\begin{lem}
\label{lem:reduced}
$\gamma_\infty$ is an efficient path in $\overline{[\gamma]}$.
\end{lem}

\begin{proof}
This follows from \cref{prop:efficient characterization}: if $L$ is any line intersecting $\Omega$ which has a lift $\h{L}$ such that $\h{\gamma}_\infty([0,1]) \cap \h{L}$ is not connected, then by density of the family $\mathcal{L}$ it is easy to see there must be a line $L' \in \mathcal{L}$ close to $L$, with a lift $\h{L}'$ close to $\h{L}$, such that $\h{\gamma}_\infty([0,1]) \cap \h{L}'$ is also not connected.  But if $\h{L}' = \h{L}_i$ in the enumeration of $\h{\mathcal{L}}$, then for all $j > i$, $\h{\gamma}_j([0,1]) \cap \h{L}'$ is connected according to \cref{lem:connected intersection}.  Since $\h{\gamma}_i \to \h{\gamma}_\infty$, the same must be true for $\h{\gamma}_\infty$, a contradiction.
\end{proof}

This completes the proof of \cref{thm:main1}.

\section{Proof of \cref{thm:main2}}
\label{sec:proof main2}

Having established the existence of a unique efficient path $\lambda$ in $\overline{[\gamma]}$, we now prove \cref{thm:main2} using \cref{thm:main1} and its proof.  In particular, we observe that in the construction of the efficient path in the proof of \cref{thm:main1}, we started with a path in $[\gamma]$ (such as $\gamma$ itself) and constructed a sequence of paths converging to the efficient path, where each path in the sequence is obtained from the previous one by replacing a subpath with a straight segment.  Such a replacement will always decrease the $\length$ length of the path and also, when finite, the Euclidean length.

It is clear that any path in $\overline{[\gamma]}$ of smallest $\length$ length (or Euclidean length if finite) must be efficient, because replacing a subpath with a straight segment decreases the length.  Thus we have (3) $\Rightarrow$ (2) and (4) $\Rightarrow$ (2).

For (2) $\Rightarrow$ (3), suppose for a contradiction that there is a path $\rho \in \overline{[\gamma]}$ with strictly smaller $\length$ length than the efficient path $\lambda \in \overline{[\gamma]}$.  By continuity of the function $\length$, it follows that there exists $\rho' \in [\gamma]$ whose $\length$ length is also strictly smaller than $\lambda$.  If we follow the construction in the proof of \cref{thm:main1} starting with $\rho'$ instead of $\gamma$, we would then obtain an efficient path in $\overline{[\gamma]}$ of $\length$ length strictly smaller than $\lambda$, contradicting the uniqueness of the efficient path.

For (2) $\Rightarrow$ (4), suppose for a contradiction that there is a path $\rho \in \overline{[\gamma]}$ with finite Euclidean length $E$ which is strictly smaller than the Euclidean length of the efficient path $\lambda \in \overline{[\gamma]}$.  This means there exists $\varepsilon > 0$ and values $0 = s_0 < s_1 < \cdots < s_N = 1$ such that $\sum_{i=1}^N |\lambda(s_i) - \lambda(s_{i-1})| > E + \varepsilon$.  Let $\h{\gamma}$ be a lift of $\gamma$ with endpoints $\h{p},\h{q} \in \overline{\disk}$, and let $\h{\lambda}$ and $\h{\rho}$ be lifts of $\lambda$ and $\rho$ with the same endpoints $\h{p},\h{q}$.

We will replace sections of the path $\rho$ with straight line segments to obtain a new path $\rho' \in \overline{[\gamma]}$ which goes within $\frac{\varepsilon}{2N}$ of each of the points $\lambda(s_i)$, in order.  To obtain $\rho'$, we proceed as follows.  For each $i = 1,\ldots,N-1$, consider the point $\h{\lambda}(s_i) \in \overline{\disk}$.
\begin{itemize}
\item If $\h{\lambda}(s_i) \in \disk$, by \cref{lem:stable generic} we can choose two stable points $\h{w}_1,\h{w}_2$ in a small neighborhood $\h{V}$ of $\h{\lambda}(s_i)$ which projects one-to-one under $\varphi$ to a small disk centered at $\lambda(s_i)$ of radius less than $\frac{\varepsilon}{2N}$, so that $\h{w}_1$ and $\h{w}_2$ are on opposite sides of $\h{\lambda}([0,1]) \cap \h{V}$.  Let $\h{L}_1$ and $\h{L}_2$ be non-separating lifts of lines such that $\h{w}_1 \in \shadow(\h{L}_1)$ and $\h{w}_2 \in \shadow(\h{L}_2)$.  By \cref{prop:efficient characterization}, $\h{\lambda}$ does not cross $\h{L}_1$ or $\h{L}_2$.  It follows that if we replace the section of $\h{\rho}$ between its first and last intersections with $\h{L}_1$ with an arc in $\h{L}_1$, and likewise for $\h{L}_2$, then the resultant path must meet $\h{V}$.
\item If $\h{\gamma}(s_i) \in \partial \disk$, choose an arc $\h{A}$ such that $\overline{\varphi}(\h{A})$ has diameter less than $\frac{\varepsilon}{2N}$, with endpoints $\h{\lambda}(s_i)$ and a stable point $\h{w} \in \disk$.  Let $\h{L}$ be a lift of a line such that $\h{w} \in \shadow(\h{L})$.  By \cref{prop:efficient characterization}, $\h{\lambda}$ does not cross $\h{L}$.  It follows that if we replace the section of $\h{\rho}$ between its first and last intersections with $\h{L}$ with an arc in $\h{L}$, then the resultant path must meet $\h{A}$.
\end{itemize}

After making the finitely many replacements of subpaths of $\h{\rho}$ with segments of lifts of lines as described above, the resultant path $\h{\rho}'$ is such that $\rho' = \overline{\varphi} \circ \h{\rho}'$ is continuous, and $\rho' \in \overline{[\gamma]}$ by \cref{thm:class closure lift}.  Clearly the Euclidean length of $\rho'$ is not greater than $E$.  On the other hand, by construction, there are values $0 = s_0' < s_1' < \cdots < s_N' = 1$ such that $|\rho'(s_i') - \lambda(s_i)| < \frac{\varepsilon}{2N}$ for each $i = 0,\ldots,N$.  It follows that the Euclidean length of $\rho'$ is at least
\begin{align*}
\sum_{i=1}^N |\rho'(s_i') - \rho'(s_{i-1}')| &> \sum_{i=1}^N \left( |\lambda(s_i) - \lambda(s_{i-1})| - 2 \cdot \frac{\varepsilon}{2N} \right) \\
&= \left( \sum_{i=1}^N |\lambda(s_i) - \lambda(s_{i-1})| \right) - \varepsilon \\
&> (E + \varepsilon) - \varepsilon = E ,
\end{align*}
a contradiction.  Therefore, $\lambda$ has smallest Euclidean length among all paths in $\overline{[\gamma]}$.

The proof that (1) $\Rightarrow$ (2) is straightforward and left to the reader.

Finally, for (2) $\Rightarrow$ (1), let $0 < s_1 < s_2 < 1$ and consider the path $\lambda {\upharpoonright}_{[s_1,s_2]}$.  We claim that there exist $s_1',s_2'$ such that $0 < s_1' \leq s_1 < s_2 \leq s_2' < 1$ and $\h{\lambda}(s_1')$ and $\h{\lambda}(s_2')$ can be joined by a path in $\overline{\disk}$ which projects under $\overline{\varphi}$ to a path of finite Euclidean length.  We rely on the fact that any two points $\h{w}_1,\h{w}_2 \in \disk$ can be joined by a path in $\disk$ which projects to a path of finite Euclidean length (e.g.\ a piecewise linear path).  If $\h{\lambda}([0,s_1]) \cap \disk \neq \emptyset$, then let $0 < s_1' \leq s_1$ be such that $\h{\lambda}(s_1') \in \disk$, and let $\h{w}_1 = \h{\lambda}(s_1')$.  If $\h{\lambda}([0,s_1]) \subset \partial \disk$, we choose $s_1'$ and $\h{w}_1$ as follows.  Note that in this case $\lambda(s_1)$ is not isolated in $\partial \Omega$, so we can apply \cref{thm:crosscut} to obtain a component $\h{C}_1$ of $\varphi^{-1}(\partial B(\lambda(s_1),\varepsilon))$, for some sufficiently small $\varepsilon > 0$, whose closure is an arc which separates $\h{\lambda}(s_1)$ from $\h{p}$ in $\overline{\disk}$.  In particular, there exists $0 < s_1' < s_1$ such that $\h{\lambda}(s_1')$ is in the closure of $\h{C}_1$.  Let $\h{w}_1$ be any point in $\h{C}_1$.  Thus, we have established that there exists $0 < s_1' \leq s_1$ and $\h{w}_1 \in \disk$ such that either $\h{\lambda}(s_1') = \h{w}_1$ or $\h{\lambda}(s_1')$ can be joined to $\h{w}_1$ by a path in $\overline{\disk}$ which projects to a circular arc.  By the same reasoning, there exists $s_2 \leq s_2' < 1$ and $\h{w}_2 \in \disk$ such that either $\h{\lambda}(s_2') = \h{w}_2$ or $\h{\lambda}(s_2')$ can be joined to $\h{w}_2$ by a path in $\overline{\disk}$ which projects to a circular arc.  It follows that $\h{\lambda}(s_1')$ can be joined to $\h{\lambda}(s_2')$ by a path in $\overline{\disk}$ which projects to a path of finite Euclidean length.  Therefore, by \cref{cor:locally shortest well-defined}, there exists a path of finite Euclidean length in $\overline{[\lambda^\gamma_{[s_1',s_2']}]}$.

Since $\lambda$ is efficient, it follows from \cref{prop:efficient characterization} that $\lambda {\upharpoonright}_{[s_1',s_2']}$ is the efficient path in $\overline{[\lambda^\gamma_{[s_1',s_2']}]}$, hence has smallest Euclidean length in this class according to the implication (2) $\Rightarrow$ (4).  In particular, $\lambda {\upharpoonright}_{[s_1',s_2']}$ has finite Euclidean length in light of the previous paragraph.  This in turn implies that the subpath $\lambda {\upharpoonright}_{[s_1,s_2]}$ has finite Euclidean length.  Again by \cref{prop:efficient characterization}, we have that $\lambda {\upharpoonright}_{[s_1,s_2]}$ is the efficient path in $\overline{[\lambda^\gamma_{[s_1,s_2]}]}$, hence has smallest Euclidean length in this class according to the implication (2) $\Rightarrow$ (4), as required.  Therefore, $\lambda$ is locally shortest.

\bibliographystyle{amsplain}
\bibliography{ShortestPath}

\end{document}